\documentclass[a4paper]{amsart}
\usepackage[utf8]{inputenc}
\usepackage[active]{srcltx}
\usepackage{verbatim}

\usepackage{esint,graphicx,color,epsfig,latexsym,longtable,amsmath,amscd,amsxtra,amssymb,amsxtra,exscale,xy,mathrsfs,wrapfig}

\usepackage{hyperref}

\usepackage[fleqn,tbtags]{mathtools}
\mathtoolsset{showonlyrefs,showmanualtags}

\theoremstyle{plain}
\newtheorem{theorem}{Theorem}[section]
\newtheorem{corollary}[theorem]{Corollary}
\newtheorem{lemma}[theorem]{Lemma}
\newtheorem{proposition}[theorem]{Proposition}
\theoremstyle{remark}

\newtheorem{definition}[theorem]{Definition}

\newtheorem{remark}[theorem]{Remark}

\numberwithin{equation}{section}

\begin{document}

\allowdisplaybreaks

\title[Eigenvalues of the Steklov $(p,2)$-Laplacian]
{Generalized eigenvalues of the $\boldsymbol{(p,2)}$-Laplacian under a parametric boundary condition}

\author{Jamil Abreu}

\address{Departamento de Matem\'atica\\
Universidade Federal de S\~ao Carlos, Rodovia Washington Lu\'is \\ S\~ao Carlos SP \\ Brazil}
\email{jamil@dm.ufscar.br}

\author{Gustavo Madeira}

\address{Departamento de Matem\'atica\\
Universidade Federal de S\~ao Carlos, Rodovia Washington Lu\'is \\ S\~ao Carlos SP \\ Brazil}
\email{gfmadeira@dm.ufscar.br}

\thanks{The first author is supported by PNPD/CAPES}



\date\today

\begin{abstract}
In this paper we study a general eigenvalue problem for the so called $(p,2)$-Laplace operator on a smooth bounded domain 
$\Omega\subset \mathbb{R}^N$ under a nonlinear Steklov type boundary condition, namely
\begin{equation}
\left\{
\begin{aligned}
-\Delta_pu-\Delta u                                 & =\lambda a(x)u  \ \ \text{in}\ \Omega,\\
(|\nabla u|^{p-2}+1)\dfrac{\partial u}{\partial\nu} & =\lambda b(x)u \ \ \text{on}\ \partial\Omega .
\end{aligned}
\right.
\end{equation}
For positive weight functions $a$ and $b$ satisfying appropriate integrability and boundedness assumptions, we show that, for all $p>1$, the 
eigenvalue set consists of an isolated null eigenvalue plus a continuous family of eigenvalues located away from zero. 
\end{abstract}

\maketitle

\section{Introduction and main results}

The spectrum of the Laplacian operator under Dirichlet as well as Neumann boundary conditions has a simple description which mathematicians usually 
learn at an early stage of their education. Consider, for instance, the case of Neumann boundary conditions. It can be inferred from a small amount of 
spectral theory that the set of all $\lambda\in \mathbb{R}$ for which there exists a non-zero $u\in W^{1,2}(\Omega)$ such that 
\begin{equation}\label{eq:eigenvalue_problem_laplacian}
\left\{
\begin{aligned}
-\Delta u                       & =\lambda u  \ \ \text{in}\ \Omega,\\
\dfrac{\partial u}{\partial\nu} & =0 \ \ \text{on}\ \partial\Omega ,
\end{aligned}
\right.
\end{equation}
can be arranged in a sequence $(\lambda_n)_{n\geqslant 0}$ of nonnegative real numbers with $\lambda_0=0$ and $\lambda_n\to \infty$. This diagonal 
structure of the Laplacian seems to be classical and difficult to attribute although the use of compactness methods to this end, at least for Dirichlet 
boundary conditions, can be traced back to \cite{rellich1930satz}. Moreover, the first positive eigenvalue can be characterized from a variational point 
of view as
\begin{equation}
\lambda_1^\text{N}=\inf\Big\{\frac{\int_\Omega |\nabla u|^2\, dx}{\int_\Omega u^2 \, dx}:u\in W^{1,2}(\Omega)\backslash \{0\}, 
\int_\Omega u\, dx =0\Big\}.
\end{equation}
For this particular result and an in-depth study of eigenvalue problems for the Laplacian we refer the interested reader to 
\cite{henrot2006extremum}.

Nonlinear eigenvalue problems for the $p$-Laplacian, that is, problems of the form \eqref{eq:eigenvalue_problem_laplacian} with $\Delta$ replaced by the 
$p$-Laplace operator $\Delta_pu=\text{div}\, (|\nabla u|^{p-2}\nabla u)$ have been also extensively studied over the past decades, see e.g. 
\cite{anane1996second}, \cite{drabek1999resonance}, \cite{friedlander1989asymptotic}, \cite{garcia1987existence}, 
\cite{le2006eigenvalue}, \cite{robinson2003average} and references therein. Most investigations rely on variational methods which usually provide 
existence of a principal eigenvalue through minimization of suitable functionals. In \cite{le2006eigenvalue} eigenvalue problems for the $p$-Laplacian subjected to different boundary conditions are studied through 
an unified treatment. It is shown, in particular, that the existence of a sequence as above having a principal eigenvalue which is simple and isolated 
from the remaining (closed) set of eigenvalues hold for the $p$-Laplacian under Dirichlet, Neumann, Robin and Steklov boundary conditions. See also 
\cite[Chapter 9]{motreanu2014topological} and the survey article \cite{drabek2007p} for further information.

In this paper we consider an eigenvalue problem for the $(p,2)$-Laplace operator
\begin{equation}\label{eq:eigenvalue_problem-p2_laplacian}
\left\{
\begin{aligned}
-\Delta_pu-\Delta u                                 & =\lambda a(x)u  \ \ \text{in}\ \Omega,\\
(|\nabla u|^{p-2}+1)\dfrac{\partial u}{\partial\nu} & =\lambda b(x)u \ \ \text{on}\ \partial\Omega ,
\end{aligned}
\right.
\end{equation}
under a nonlinear Steklov boundary condition, that is, a boundary condition which is itself an eigenvalue problem, usually known in the linear case as Steklov eigenvalue problem, since its first appearance in 
\cite{stekloff1902problemes}. Here $\Omega\subset \mathbb{R}^N$ is a bounded domain with smooth boundary and $\nu$ stands for the outward unit normal 
to $\partial\Omega$. Moreover, $a$ and $b$ are given nonnegative functions on $\Omega$ and $\partial\Omega$, respectively, satisfying certain 
integrability conditions and
\begin{equation}\label{eq:int-a+int-b_maior_zero}
\int_\Omega a(x)\, dx+\int_{\partial\Omega} b(x)\, d\sigma > 0 .
\end{equation}
By reflection, this cover the case where both functions are negative with at least one of them being strictly negative 
in a set with positive measure. 

The operator $-\Delta_p-\Delta$ appears e.g. in quantum field theory \cite{benci2000solitons}. From a mathematical point of view it presents several 
difficulties due to its nonhomogeneity. Elliptic equations involving such an operator have been extensively studied over the last years; for instance, 
resonance and existence of nodal solutions for such equations is a current research topic, see e.g. \cite{aizicovici2015nodal}, 
\cite{papageorgiou2015resonantreaction}, \cite{papageorgiou2015resonantconcave} and references therein. Problem \eqref{eq:eigenvalue_problem-p2_laplacian} 
with $a\equiv 1$ and $b\equiv 0$ (Neumann boundary condition) has been studied recently in \cite{farcaseanu2015set} (in the case 
$1<p<2$) and \cite{mihailescu2011eigenvalue} (case $p>2$). Note that condition \eqref{eq:int-a+int-b_maior_zero} is trivially satified in this case. 
These authors have shown that the generalized spectrum for this problem is of `point plus continuum' type, that is, the eigenvalue set 
consists of a zero eigenvalue plus an unbounded open interval with starting point away 
from zero. In particular, there exists a principal eigenvalue but the set of eigenvalues is not closed. In this paper we push 
their analysis further and show that a `point plus continuum' spectrum still holds in a much more general setting (although, probably, not the most
general one, something that will be investigated elsewhere). Many authors have worked on 
eigenvalue problems for the $(p,2)$-Laplacian (and more generally, for the $(p,q)$-Laplacian), most of them under Dirichlet boundary 
conditions, see e.g. \cite{barile2015some}, \cite{tanaka2014generalized} and references therein. To the best of our knowledge the present 
work is the first one dealing with the (generalized) spectrum of the $(p,2)$-Laplacian under Steklov type boundary conditions. We also note that 
the techniques employed in the proof of Theorem \ref{thm:main} do not generalise to the $(p,q)$-Laplacian.

For each $p>1$ define
\begin{multline}\label{eq:first_nonzero_eigenvalue}
\lambda_1(p):=
\inf\Big\{\frac{\frac{1}{p}\int_\Omega|\nabla u|^p\, dx +\frac{1}{2}\int_\Omega |\nabla u|^2\, dx}
{\frac{1}{2}\int_\Omega a(x)u^2 \, dx +\frac{1}{2}\int_{\partial\Omega} b(x)u^2 \, d\sigma}:u\in W^{1,p}(\Omega)\backslash \{0\}, \\
 \int_\Omega a(x) u\, dx+\int_{\partial\Omega} b(x) u\, d\sigma =0\Big\},
\end{multline}
and
\begin{multline}\label{eq:first_nonzero_eigenvalue-W12}
\mu_1(p):=
\inf\Big\{\frac{\frac{1}{p}\int_\Omega|\nabla u|^p\, dx +\frac{1}{2}\int_\Omega |\nabla u|^2\, dx}
{\frac{1}{2}\int_\Omega a(x)u^2 \, dx +\frac{1}{2}\int_{\partial\Omega} b(x)u^2 \, d\sigma}:u\in W^{1,2}(\Omega)\backslash \{0\}, \\
 \int_\Omega a(x) u\, dx+\int_{\partial\Omega} b(x) u\, d\sigma =0\Big\}. 
\end{multline}

We can now state our main result.

\begin{theorem}\label{thm:main}
Let $\Omega\subset \mathbb{R}^N$ be a bounded domain with smooth boundary, $N>2$. Suppose $a$ and $b$ are non-negative measurable functions on 
$\Omega$ and $\partial\Omega$, respectively, and satisfying condition \eqref{eq:int-a+int-b_maior_zero}. Let $\lambda_1(p)$ and $\mu_1(p)$ be the 
numbers defined in Eqs. \eqref{eq:first_nonzero_eigenvalue} and \eqref{eq:first_nonzero_eigenvalue-W12}, respectively.
\begin{enumerate}
\item \label{item:p_bigger_2} If $p>2$, $a\in L^{\frac{N}{2}}(\Omega)$ and $b\in L^{N-1}(\partial\Omega)$ then the set 
of eigenvalues of Problem \eqref{eq:eigenvalue_problem-p2_laplacian} equals  $\{0\}\cup (\lambda_1(p),\infty)$.
\item \label{item:p_less_2} If $1<p<2$ and
\begin{enumerate}
\item \label{item:p_less_2A} either $\frac{2N}{N+1}<p<2$, $a\in L^{\frac{pN}{(p-2)N+2p}}(\Omega)$ and 
$b\in L^{\frac{p(N-1)}{(p-2)N+p}}(\partial\Omega)$,
\item \label{item:p_less_2B} or $\frac{2N}{N+2}<p\leqslant\frac{2N}{N+1}$, $a\in L^{\frac{pN}{(p-2)N+2p}}(\Omega)$ and 
$b\in L^{\infty}(\partial\Omega)$,
\item \label{item:p_less_2C} or $1<p\leqslant\frac{2N}{N+2}$, $a\in L^{\infty}(\Omega)$ and $b\in L^{\infty}(\partial\Omega)$,
\end{enumerate}
then the set of eigenvalues of Problem \eqref{eq:eigenvalue_problem-p2_laplacian} equals 
$\{0\}\cup (\mu_1(p),\infty)$.
\end{enumerate}
\end{theorem}

Later we will be able to find simpler expressions for the numbers $\lambda_1(p)$ and $\mu_1(p)$, cf. Eqs. \eqref{eq:first_eigenvalue-nu} and 
\eqref{eq:nu1}; from Theorem \ref{thm:main}\ref{item:p_less_2} and Eq. \eqref{eq:nu1} we find that the spectrum of the $(p,2)$-Laplacian under the 
conditions stated in Assertion \ref{item:p_less_2C} above actually does not depend on $p$. Assertions \ref{item:p_bigger_2} and \ref{item:p_less_2} are treated with 
different techniques. Section \ref{sec:proof-assertion_a} is devoted to the 
proof of Theorem \ref{thm:main}\ref{item:p_bigger_2} and is based on a standard procedure of associating a weakly lower semicontinuous functional to 
Problem \eqref{eq:eigenvalue_problem-p2_laplacian}. In Section \ref{sec:proof-assertion_b} we carry out the proof of Theorem 
\ref{thm:main}\ref{item:p_less_2} based on minimization over the associated Nehari manifold. 

In the following corollary (with $a\equiv 1$ and $b\equiv 0$) we recover the mains results in \cite{farcaseanu2015set} and \cite{mihailescu2011eigenvalue}.
\begin{corollary}
Let $\Omega\subset \mathbb{R}^N$ be a bounded domain with smooth boundary. Suppose $0\leqslant a\in L^{\infty}(\Omega)$ and 
$0\leqslant b\in L^{\infty}(\partial\Omega)$ are given functions satisfying condition \eqref{eq:int-a+int-b_maior_zero}.
\begin{enumerate}
 \item If $p>2$ then the set 
 of eigenvalues of Problem \eqref{eq:eigenvalue_problem-p2_laplacian} is given by  $\{0\}\cup (\lambda_1(p),\infty)$ where $\lambda_1(p)$ is the 
 number defined in Eq. \eqref{eq:first_nonzero_eigenvalue}.
 \item If $1<p<2$ then the set 
 of eigenvalues of Problem \eqref{eq:eigenvalue_problem-p2_laplacian} is given by  $\{0\}\cup (\mu_1(p),\infty)$ where $\mu_1(p)$ is the number 
 defined in Eq. \eqref{eq:first_nonzero_eigenvalue-W12}.
\end{enumerate}
Moreover, the eigenvalue set of Problem \eqref{eq:eigenvalue_problem-p2_laplacian} does not depend on $p$ when $1<p<2$.
\end{corollary}

We observe that our Theorem \ref{thm:main} is not valid with $p=2$; in this case, Problem \eqref{eq:eigenvalue_problem-p2_laplacian} reduces to 
\begin{equation}\label{eq:eigenvalue_problem-p2_laplacian-p_igual_2}
\left\{
\begin{aligned}
-\Delta u                                 & =\dfrac{\lambda}{2} a(x)u  \ \ \text{in}\ \Omega,\\
\dfrac{\partial u}{\partial\nu} & =\frac{\lambda}{2} b(x)u \ \ \text{on}\ \partial\Omega ,
\end{aligned}
\right.
\end{equation}
which is a Steklov problem for the Laplacian whose spectrum has the well known structure described earlier in this introduction.

Let us finish this introduction by explaining the role of the various integrability assumptions on $a$ and $b$. These hypotheses 
are directly related to the well known embeddings $W^{1,r}(\Omega)\hookrightarrow L^q(\Omega)$ which holds in the cases: 
(i) $1\leqslant q\leqslant r^*=\frac{rN}{N-r}$, if $1\leqslant r<N$; (ii) $r\leqslant q<\infty$, if $r=N$; (iii) $q=\infty$, if $r>N$. Moreover, these 
embeddings are compact when $1\leqslant q<r^*$ in case (i), all $q$ in case (ii), and in case (iii) when reinterpreted as 
$W^{1,r}(\Omega)\hookrightarrow C(\overline{\Omega})$. We also have trace embeddings 
$W^{1,r}(\Omega)\hookrightarrow L^q(\partial\Omega)$ for all $1\leqslant r\leqslant q\leqslant \frac{r(N-1)}{N-r}$ if $1\leqslant r<N$ and similarly as 
before in the other ranges of $r$. Details can be found in the standard literature, see e.g. \cite[Chapter 5]{adams2003sobolev} 
or \cite[Chapter 9]{brezis2011functional}. 

\begin{remark}
We can take $N=2$ in Theorem \ref{thm:main}. This is clear with regard to Item \ref{item:p_less_2} and requires small modifications in 
\ref{item:p_bigger_2}. To be precise, we can consider the embedding $W^{1,2}(\Omega)\hookrightarrow L^{q}(\Omega)$ with any $q>2$ and assume that 
$a\in L^{\frac{q}{q-2}}(\Omega)$; if we think of large $q$'s this means that we can take $a\in L^{1+\delta}(\Omega)$ for any $\delta>0$. Similar 
considerations applies to the trace embedding and the corresponding integrability assumptions on $b$. 
\end{remark}

\section{Proof of Theorem \ref{thm:main}\ref{item:p_bigger_2}}\label{sec:proof-assertion_a}

If $p>2$ we have $W^{1,p}(\Omega)\hookrightarrow W^{1,2}(\Omega)$ and it is natural to consider solutions in $W^{1,p}(\Omega)$. For our purposes it will 
be convenient to consider the embeddings $W^{1,r}(\Omega)\hookrightarrow L^{\frac{rN}{N-r}}(\Omega)$ and 
$W^{1,r}(\Omega)\hookrightarrow L^{\frac{r(N-1)}{N-r}}(\partial\Omega)$ with $r=2$. In this case, if $a\in L^{\frac{N}{2}}(\Omega)$ and 
$b\in L^{N-1}(\Omega)$ then integrals such as $\int_\Omega a(x)u^2\, dx$ and $\int_{\partial\Omega} b(x)u^2\, d\sigma$ will be well-defined and 
good estimates can be obtained. Moreover, we must restrict to dimensions $N>2$, which we we assume throughout this section. 

In order to find the Euler-Lagrange equation, and the energy functional, associated to Problem 
\eqref{eq:eigenvalue_problem-p2_laplacian} we formally multiply it by a smooth function $\phi$ to obtain
\begin{align*}
&\lambda \int_\Omega a(x)u\phi \, dx\\
&=-\int_\Omega {\rm div}\,(|\nabla u|^{p-2}\nabla u)\phi\, dx -\int_\Omega (\Delta u)\phi\, dx\\
&=\int_\Omega |\nabla u|^{p-2}\nabla u\cdot \nabla\phi\, dx-\int_{\partial\Omega}|\nabla u|^{p-2}\partial_\nu u\phi\, d\sigma
+\int_\Omega\nabla u\cdot \nabla \phi\, dx-\int_{\partial\Omega}\partial_\nu u\phi\, d\sigma\\
&=\int_\Omega |\nabla u|^{p-2}\nabla u\cdot \nabla\phi\, dx+\int_\Omega\nabla u\cdot \nabla \phi\, dx
-\lambda\int_{\partial\Omega} b(x)\phi\, d\sigma .
\end{align*}
This computation lead us naturally to the following Definition.

\begin{definition}\label{def:eigenvalue-eigenfunction-p2_laplacian}
Let $p>2$. We call $\lambda\in \mathbb{R}$ an {\it eigenvalue} of Problem \eqref{eq:eigenvalue_problem-p2_laplacian} if there exists a non-zero 
$u\in W^{1,p}(\Omega)$ such that
\begin{equation}\label{eq:eigenvalue-eigenfunction-p2_laplacian}
\int_\Omega |\nabla u|^{p-2}\nabla u\cdot \nabla \phi \, dx +\int_\Omega \nabla u\cdot \nabla \phi \, dx =
\lambda \int_\Omega a(x)u\phi \, dx +\lambda \int_{\partial\Omega} b(x)u \phi \, d\sigma
\end{equation}
for all $\phi\in W^{1,p}(\Omega)$. Such a function $u\in W^{1,p}(\Omega)$ will be called an {\it eigenfunction} corresponding to the eigenvalue $\lambda$. 
In other words, $\lambda\in \mathbb{R}$ is an eigenvalue of Problem \eqref{eq:eigenvalue_problem-p2_laplacian} with corresponding eigenfunction 
$u\in W^{1,p}(\Omega)\backslash\{0\}$ if and only if $u$ is a critical point of the $C^1$ functional 
\begin{equation}\label{eq:functional-p2_laplacian}
\mathscr{I}_\lambda(u)=\frac{1}{p}\int_\Omega|\nabla u|^p\, dx +\frac{1}{2}\int_\Omega |\nabla u|^2\, dx 
-\frac{\lambda}{2}\int_\Omega a(x)u^2 \, dx -\frac{\lambda}{2} \int_{\partial\Omega} b(x)u^2 \, d\sigma . 
\end{equation}
\end{definition}

It is well known that the Sobolev space $W^{1,p}(\Omega)$ can be decomposed as a direct sum
\begin{equation}\label{eq:W1p-igual-Vp+R}
W^{1,p}(\Omega)=\mathscr{V}_p\oplus \mathbb{R}, 
\end{equation}
where $\mathscr{V}_p$ is the closed subspace consisting of all mean zero elements in $W^{1,p}(\Omega)$, that is,
$$
\mathscr{V}_p:=\Big\{u\in W^{1,p}(\Omega):\int_\Omega u\, dx =0\Big\}.
$$
One of the main advantages of the decomposition \eqref{eq:W1p-igual-Vp+R} relies on the fact that, for elements in $\mathscr{V}_p$, 
Poincar\'e-Wirtinger inequality takes its simplest form, namely
\begin{equation}
\int_\Omega|u|^p\, dx\leqslant C^\text{P}_p\int_{\Omega}|\nabla u|^p\, dx\ \ \ (u\in \mathscr{V}_p). 
\end{equation}

For our purposes, however, it will be convenient to introduce another decomposition. Let 
\begin{equation}
\mathscr{W}_p:=\Big\{u\in W^{1,p}(\Omega):\int_\Omega a(x)u\, dx+\int_{\partial\Omega} b(x)u\, d\sigma =0\Big\}.
\end{equation}

\begin{lemma}\label{lemma:W1p-igual-Vptil+R}
Let $p> 2$. Then $\mathscr{W}_p$ is a closed subspace of $W^{1,p}(\Omega)$ and we have the decomposition
\begin{equation}\label{eq:W1p-igual-Vptil+R}
W^{1,p}(\Omega)=\mathscr{W}_p\oplus \mathbb{R}.
\end{equation}
\end{lemma}

\begin{proof}
Let $\varphi:W^{1,p}(\Omega)\to \mathbb{R}$ be defined by $\varphi(u)=\int_\Omega a(x)u\, dx+\int_{\partial\Omega} b(x)u\, d\sigma $. Then
\begin{align*}
|\varphi(u)|&\leqslant \Big(\int_\Omega a(x)\, dx\Big)^{\frac{1}{2}}\Big(\int_\Omega a(x)u^2\, dx\Big)^{\frac{1}{2}}+
\Big(\int_{\partial\Omega} b(x)\, d\sigma\Big)^{\frac{1}{2}}\Big(\int_{\partial\Omega} b(x)u^2\, d\sigma\Big)^{\frac{1}{2}}\\
& \leqslant \widetilde{C}\Big(\int_\Omega a(x)u^2\, dx+\int_{\partial\Omega} b(x)u^2\, d\sigma\Big)^{\frac{1}{2}}
\end{align*}
with 
$\widetilde{C}=\sqrt{2}\max\big\{\big(\int_\Omega a(x)\, dx\big)^{\frac{1}{2}} ,\big(\int_{\partial\Omega} b(x)\, d\sigma\big)^{\frac{1}{2}}\big\}$. 
We have
$$
\int_\Omega a(x)u^2\, dx\leqslant \|a\|_{L^{\frac{N}{2}}(\Omega)}\Big[\Big(\int_\Omega|u|^{\frac{2N}{N-2}}\, dx\Big)^{\frac{N-2}{2N}}\Big]^2
\leqslant C_1\|a\|_{L^{\frac{N}{2}}(\Omega)}\|u\|_{W^{1,2}(\Omega)}^2
$$
and 
\begin{multline}
\int_{\partial\Omega} b(x)u^2\, d\sigma\leqslant \|b\|_{L^{N-1}(\partial\Omega)}
\Big[\Big(\int_{\partial\Omega}|u|^{\frac{2(N-1)}{N-2}}\, d\sigma\Big)^{\frac{N-2}{2(N-1)}}\Big]^2\\
\leqslant C_2\|b\|_{L^{N-1}(\partial\Omega)}\|u\|_{W^{1,2}(\Omega)}^2.
\end{multline}
Here $C_1$ and $C_2$ are the Sobolev and trace constants for the embeddings mentioned in the beginning of this section. Thus $\varphi$ 
belongs to $\big(W^{1,p}(\Omega)\big)^*$ and then $\mathscr{W}_p=\ker \varphi$ is a closed hyperplane. Moreover, condition 
\eqref{eq:int-a+int-b_maior_zero} implies that constant functions lie outside of $\mathscr{W}_p$. This proves the decomposition 
\eqref{eq:W1p-igual-Vptil+R}.
\end{proof}

\begin{remark}\label{rmk:eigenfunction_lies_on_Vptil}
If $u$ is an eigenfunction corresponding to a non-zero eigenvalue then, by testing Eq. \eqref{eq:eigenvalue-eigenfunction-p2_laplacian} 
against a constant function, we find that $u\in \mathscr{W}_p$. This is the main motivation for introducing the space $\mathscr{W}_p$.
\end{remark}

We observe that, with the notations just introduced, the definition of $\lambda_1(p)$ in Eq. \eqref{eq:first_nonzero_eigenvalue} can be reformulated 
as
\begin{equation}\label{eq:first_nonzero_eigenvalue-Vtil}
\lambda_1(p):=
\inf_{u\in \mathscr{W}_p\backslash \{0\}}\frac{\frac{1}{p}\int_\Omega|\nabla u|^p\, dx +\frac{1}{2}\int_\Omega |\nabla u|^2\, dx}
{\frac{1}{2}\int_\Omega a(x)u^2 \, dx +\frac{1}{2}\int_{\partial\Omega} b(x)u^2 \, d\sigma}. 
\end{equation}

The proof of Theorem \ref{thm:main}\ref{item:p_bigger_2} will follow as a consequence of several intermediate results, most of them being of independent 
interest. The following elementary result already establishes almost half of our main result. Although we state it under the assumption that $p>2$, to be consistent 
with Definition \ref{def:eigenvalue-eigenfunction-p2_laplacian}, the reader will notice that all arguments would work quite well for all $p>1$. We 
will need this later.

\begin{lemma}\label{lemma:eigenvalue_is_nonegative}
Let $p>2$.
\begin{enumerate}
\item\label{item:0-eigenvalue} $\lambda =0$ is an eigenvalue of Problem \eqref{eq:eigenvalue_problem-p2_laplacian}.
\item\label{item:negative-not-eigenvalue} No number $\lambda<0$ is an eigenvalue of Problem \eqref{eq:eigenvalue_problem-p2_laplacian}.
\end{enumerate}
\end{lemma}

\begin{proof}
Assertion \ref{item:0-eigenvalue} is immediate since Eq. \eqref{eq:eigenvalue-eigenfunction-p2_laplacian} is obviously satisfied when $\lambda =0$ 
and $u$ is a constant function. To prove Assertion \ref{item:negative-not-eigenvalue}, suppose $\lambda$ is a nonzero eigenvalue with corresponding 
eigenfunction $u_\lambda$. By testing Eq. \eqref{eq:eigenvalue-eigenfunction-p2_laplacian} against $\phi=u_\lambda$ yields
\begin{equation}
\int_\Omega|\nabla u_\lambda|^p\, dx+\int_\Omega|\nabla u_\lambda|^2\, dx=
\lambda\Big(\int_\Omega a(x)u_\lambda^2\, dx+\int_{\partial\Omega} b(x)u_\lambda^2\, d\sigma\Big)
\end{equation}
thus $\lambda > 0$. This shows that no eigenvalue can be strictly negative.
\end{proof}

\begin{remark}[Null eigenvalues versus constant eigenvectors]
Let us clarify the (easy) relation between null eigenvalues and constant eigenvectors, which appears in the proof above. On the one hand, if Eq. 
\eqref{eq:eigenvalue-eigenfunction-p2_laplacian} is satisfied by $\lambda =0$, some $u\in W^{1,p}(\Omega)\backslash \{0\}$ and all 
$\phi\in W^{1,p}(\Omega)$ then (by testing $\phi=u$) we find that $u$ is constant (by Poincaré inequality). On the other hand, if Eq. 
\eqref{eq:eigenvalue-eigenfunction-p2_laplacian} is satisfied by some $\lambda\in \mathbb{R}$ (we can take $\lambda\geqslant 0$, by Lemma 
\ref{lemma:eigenvalue_is_nonegative}), some non-zero constant $u$ and all $\phi\in W^{1,p}(\Omega)$ then (again by testing $\phi=u$) we find that $\lambda=0$.
\end{remark}

\begin{lemma}\label{lemma:lambda1_positivo}
$\lambda_1(p)>0$ for all $p>2$.
\end{lemma}

\begin{proof}
We claim that
\begin{equation}\label{eq:estimate-integral-of-au2+bu2}
 \int_\Omega a(x)u^2\, dx+\int_{\partial\Omega} b(x)u^2\, d\sigma \leqslant \int_\Omega a(x)(u-\overline{u})^2\, dx
 +\int_{\partial\Omega} b(x)(u-\overline{u})^2\, d\sigma,
\end{equation}
for all $u\in \mathscr{W}_p$, where $\overline{u}=\frac{1}{|\Omega|}\int_{\Omega}u\, dx$. To see this, write $u=(u-\overline{u})+\overline{u}$ and note 
that $u^2\leqslant (u-\overline{u})^2+2u\overline{u}$, thus by integrating we find
\begin{multline}
\int_\Omega a(x)u^2\, dx+\int_{\partial\Omega} b(x)u^2\, d\sigma \\ 
\leqslant \int_\Omega a(x)(u-\overline{u})^2\, dx+\int_{\partial\Omega} b(x)(u-\overline{u})^2\, d\sigma \\
+2\overline{u}\Big(\int_\Omega a(x)u\, dx+\int_{\partial\Omega} b(x)u\, d\sigma\Big)
\end{multline}
which gives the estimate, since the last summand vanishes.

It follows from estimate \eqref{eq:estimate-integral-of-au2+bu2}, in combination with the estimates obtained in the proof of 
Lemma \ref{lemma:W1p-igual-Vptil+R}, that 
\begin{multline}
\int_\Omega a(x)u^2\, dx+\int_{\partial\Omega} b(x)u^2\, d\sigma \\
\leqslant \big(C_1\|a\|_{L^\frac{N}{2}(\Omega)}+C_2\|b\|_{L^{N-1}(\partial\Omega)}\big)\|u-\overline{u}\|_{W^{1,2}(\Omega)}^2\\
\leqslant \big(C_1\|a\|_{L^\frac{N}{2}(\Omega)}+C_2\|b\|_{L^{N-1}(\partial\Omega)}\big)(1+C^\text{P}_2)\int_{\Omega}|\nabla u|^2\, dx 
\end{multline}
for all $u\in \mathscr{W}_p$, where $C^\text{P}_2$ is the constant in Poincaré-Wirtinger inequality for $p=2$. Thus
$$
\frac{\frac{1}{p}\int_\Omega|\nabla u|^p\, dx +\frac{1}{2}\int_\Omega |\nabla u|^2\, dx}
{\frac{1}{2}\int_\Omega a(x)u^2 \, dx +\frac{1}{2}\int_{\partial\Omega} b(x)u^2 \, d\sigma}
> \frac{1}{\big(C_1\|a\|_{L^\frac{N}{2}(\Omega)}+C_2\|b\|_{L^{N-1}(\partial\Omega)}\big)(1+C^\text{P}_2)}
$$
for all $u\in \mathscr{W}_p\backslash \{0\}$. From this it follows immediately that $\lambda_1(p)>0$.
\end{proof}

\begin{remark}
The previous proof also gives the estimate
$$
\lambda_1(p)\geqslant \frac{1}{\big(C_1\|a\|_{L^\frac{N}{2}(\Omega)}+C_2\|b\|_{L^{N-1}(\partial\Omega)}\big)(1+C^\text{P}_2)}
$$
which gives a bound from bellow for $\lambda_1(p)$ in terms of $a$, $b$ and some Sobolev and trace embeddings constants that do not depend 
on $p$. This should not be surprising, since $p>2$.
\end{remark}

The following lemma shows, essentially, that the functional defined in Eq. \eqref{eq:functional-p2_laplacian} is coercive for $p>2$, when restricted 
to the subspace $\mathscr{W}_p$. 

\begin{lemma}\label{lemma:limit-coercive-Vptil}
Let $p>2$. For every $\lambda>0$ we have
\begin{equation}\label{eq:limit-coercive-Vptil}
\lim_{\|u\|_{W^{1,p}(\Omega)}\to \infty,\, u\in \mathscr{W}_p}\Big(\frac{1}{p}\int_\Omega|\nabla u|^p\, dx  
-\frac{\lambda}{2}\int_\Omega a(x)u^2 \, dx -\frac{\lambda}{2} \int_{\partial\Omega} b(x)u^2 \, d\sigma\Big) =\infty .
\end{equation}
\end{lemma}

First we need a technical tool.
\begin{lemma}\label{lemma:u-u_0-converges-infty}
For each $u\in W^{1,p}(\Omega)$, define $u^{(0)}=u-\overline{u}$, where $\overline{u}=\frac{1}{|\Omega|}\int_{\Omega}u\, dx$. Then, on the 
subspace $\mathscr{W}_p$, $\|u^{(0)}\|_{W^{1,p}(\Omega)}\to \infty$ as $\|u\|_{W^{1,p}(\Omega)}\to \infty$.
\end{lemma}

\begin{proof}
If the conclusion is false then there is a sequence $(u_n)\subset \mathscr{W}_p$ such that $\|u_n\|_{W^{1,p}(\Omega)}\to \infty$ for which
$\|u_n^{(0)}\|_{W^{1,p}(\Omega)}\leqslant C$ for some constant $C\geqslant 0$. Since 
$\int_\Omega |\nabla u_n|^p\, dx =\int_\Omega |\nabla u_n^{(0)}|^p\, dx\leqslant C^p$ we must have $\|u_n\|_{L^{p}(\Omega)}\to \infty$. Set 
$v_n:=u_n/\|u_n\|_{L^{p}(\Omega)}$. Then there exists 
$v_0\subset \mathscr{W}_p$ such that $v_n\rightharpoonup v_0$ in $W^{1,p}(\Omega)$ and $v_n\to v_0$ in $L^p(\Omega)$. But then
$$
\int_\Omega |\nabla v_0|^p\, dx\leqslant \liminf_{n\to \infty} \int_\Omega |\nabla v_n|^p\, dx=
\liminf_{n\to \infty} \frac{1}{\|u_n\|_{L^{p}(\Omega)}^p}\int_\Omega |\nabla u_n|^p\, dx=0,
$$
thus $v_0$ is constant. By Lemma \ref{lemma:W1p-igual-Vptil+R} this constant is zero, which contradicts the fact that $\|v_n\|_{L^p(\Omega)}=1$.
\end{proof}

\begin{proof}[Proof of Lemma \ref{lemma:limit-coercive-Vptil}]
For simplicity, let us introduce the notation $u^{(0)}=u-\overline{u}$, where $\overline{u}=\frac{1}{|\Omega|}\int_{\Omega}u\, dx$, so that 
estimate \eqref{eq:estimate-integral-of-au2+bu2} takes the form
\begin{equation}
 \int_\Omega a(x)u^2\, dx+\int_{\partial\Omega} b(x)u^2\, d\sigma 
 \leqslant \int_\Omega a(x)(u^{(0)})^2\, dx +\int_{\partial\Omega} b(x)(u^{(0)})^2\, d\sigma
 \ \ \ (u\in \mathscr{W}_p).
\end{equation}
Moreover, $\nabla u=\nabla u^{(0)}$ and $\|u^{(0)}\|_{W^{1,p}(\Omega)}\to \infty$ as $\|u\|_{W^{1,p}(\Omega)}\to \infty$ by 
Lemma \ref{lemma:u-u_0-converges-infty}. Therefore, it suffices to prove \eqref{eq:limit-coercive-Vptil} when $u\in \mathscr{V}_p$, that is
\begin{equation}\label{eq:limit-coercive-Vp}
\lim_{\|u\|_{W^{1,p}(\Omega)}\to \infty,\, u\in \mathscr{V}_p}\Big(\frac{1}{p}\int_\Omega|\nabla u|^p\, dx  
-\frac{\lambda}{2}\int_\Omega a(x)u^2 \, dx -\frac{\lambda}{2} \int_{\partial\Omega} b(x)u^2 \, d\sigma\Big) =\infty .
\end{equation}

Now, we have $\int_\Omega|\nabla u|^p\, dx\geqslant \frac{1}{1+C^\text{P}_p}\|u\|^p_{W^{1,p}(\Omega)}$ by Poincaré-Wirtinger 
inequality and the terms $\int_\Omega a(x)u^2 \, dx$ and $\int_{\partial\Omega} b(x)u^2 \, d\sigma$ can be both estimated (up to a multiplicative 
constant) by $\int_{\Omega}|\nabla u|^2\, dx$ which, in turn, can be estimated (up to a multiplicative constant) 
by $\|u\|^2_{W^{1,p}(\Omega)}$. Thus, up to a multiplicative constant, the expression in \eqref{eq:limit-coercive-Vptil} can be estimated by 
$\|u\|^p_{W^{1,p}(\Omega)} -\|u\|^2_{W^{1,p}(\Omega)}$. Since $p>2$, the conclusion follows.
\end{proof}

\begin{proposition}\label{prop:eigenvalues_bigger_than_lambda1}
Let $p>2$. Every number $\lambda\in (\lambda_1(p),\infty)$ is an eigenvalue of Problem \eqref{eq:eigenvalue_problem-p2_laplacian}.
\end{proposition}

\begin{proof}
Fix $\lambda\in (\lambda_1(p),\infty)$ and define $\mathscr{I}_\lambda:\mathscr{W}_p\to \mathbb{R}$ by \eqref{eq:functional-p2_laplacian}, that is
\begin{equation*}
\mathscr{I}_\lambda(u)=\frac{1}{p}\int_\Omega|\nabla u|^p\, dx +\frac{1}{2}\int_\Omega |\nabla u|^2\, dx 
-\frac{\lambda}{2}\int_\Omega a(x)u^2 \, dx -\frac{\lambda}{2} \int_{\partial\Omega} b(x)u^2 \, d\sigma \ \ \ (u\in \mathscr{W}_p). 
\end{equation*}
It is standard to show that $\mathscr{I}_\lambda\in C^1(\mathscr{W}_p;\mathbb{R})$ and that its derivative is given by
\begin{equation*}
\langle \mathscr{I}_\lambda'(u),\phi \rangle =\int_\Omega |\nabla u|^{p-2}\nabla u\cdot \nabla \phi \, dx +\int_\Omega \nabla u\cdot \nabla \phi \, dx 
-\lambda \int_\Omega a(x)u\phi \, dx -\lambda \int_{\partial\Omega} b(x)u \phi \, d\sigma .
\end{equation*}
It is also elementary to check that $\mathscr{I}_\lambda$ is weakly lower semicontinuous on $\mathscr{W}_p$. Moreover, Lemma \ref{lemma:limit-coercive-Vptil} implies 
that $\mathscr{I}_\lambda$ is coercive, meaning that
$$
\lim_{\|u\|_{W^{1,p}(\Omega)}\to \infty,\, u\in \mathscr{W}_p}\mathscr{I}_\lambda(u)=\infty .
$$
Standard results in the calculus of variations (cf. \cite[Theorem 1.2]{struwe1990variational}) assure the existence of a global minimum point 
$u_\lambda\in \mathscr{W}_p$ for $\mathscr{I}_\lambda$. Since $\lambda>\lambda_1(p)$, it follows from the very definition of $\lambda_1(p)$ that there is some 
$v_{\lambda}$ satisfying $\mathscr{I}_\lambda(v_\lambda)<0$. Thus $\mathscr{I}_\lambda(u_\lambda)\leqslant \mathscr{I}_\lambda(v_\lambda)<0$ and we 
can infer that $u_\lambda\neq 0$. Moreover, the obvious identity
$$\langle \mathscr{I}_\lambda'(u_\lambda),\phi \rangle =0\ \ \ (\phi\in \mathscr{W}_p)$$
is also satisfied when $\phi$ is a constant function. It follows from Lemma \ref{lemma:W1p-igual-Vptil+R} that this identity is then satisfied for 
every $\phi\in W^{1,p}(\Omega)$. Therefore $\lambda$ is an eigenvalue according to Definition \ref{def:eigenvalue-eigenfunction-p2_laplacian}. 
 \end{proof}

\begin{proposition}\label{prop:no-eigenvalue_in_interval}
Let $p>2$. No number $\lambda\in (0,\lambda_1(p))$ is an eigenvalue of Problem \eqref{eq:eigenvalue_problem-p2_laplacian}.
\end{proposition}

\begin{proof}
First, note that 
\begin{multline}
\frac{\lambda_1(p)-\lambda}{2}\Big(\int_\Omega a(x)u^2 \, dx + \int_{\partial\Omega} b(x)u^2 \, d\sigma\Big)\\
\leqslant \frac{1}{p}\int_\Omega|\nabla u|^p\, dx +\frac{1}{2}\int_\Omega |\nabla u|^2\, dx 
-\frac{\lambda}{2}\int_\Omega a(x)u^2 \, dx -\frac{\lambda}{2} \int_{\partial\Omega} b(x)u^2 \, d\sigma ,
\end{multline}
for every $u\in \mathscr{W}_p\backslash\{0\}$ and $\lambda\in \mathbb{R}$. If there was an eigenvalue $\lambda\in (0,\lambda_1(p))$ with corresponding eigenfunction 
$u_\lambda\in \mathscr{W}_p\backslash\{0\}$ then the above estimate would imply
\begin{multline}
0<\frac{\lambda_1(p)-\lambda}{2}\Big(\int_\Omega a(x)u_\lambda^2 \, dx + \int_{\partial\Omega} b(x)u_\lambda^2 \, d\sigma\Big)\\
\leqslant \frac{1}{2}\int_\Omega|\nabla u_\lambda|^p\, dx +\frac{1}{2}\int_\Omega |\nabla u_\lambda|^2\, dx 
-\frac{\lambda}{2}\int_\Omega a(x)u_\lambda^2 \, dx -\frac{\lambda}{2} \int_{\partial\Omega} b(x)u_\lambda^2 \, d\sigma=0, 
\end{multline}
where the last identity follows by testing Eq. \eqref{eq:eigenvalue-eigenfunction-p2_laplacian} against $\phi=u_\lambda$. This is obviously a 
contradiction.
\end{proof}

\begin{lemma}\label{lemma:lambda_igual_nu}
Let $p>2$. Define
\begin{equation}\label{eq:first_eigenvalue-nu}
\nu_1(p):=\inf_{u\in \mathscr{W}_p\backslash \{0\}}
\frac{\int_\Omega |\nabla u|^2\, dx}{\int_\Omega a(x)u^2 \, dx + \int_{\partial\Omega} b(x)u^2\,d\sigma}. 
\end{equation}
Then $\lambda_1(p)=\nu_1(p)$.
\end{lemma}

\begin{proof}
The estimate $\nu_1(p)\leqslant \lambda_1(p)$ is obvious. On the other hand, for each $u\in \mathscr{W}_p\backslash \{0\}$ and $t>0$ we have
\begin{align*}
\lambda_1(p)&\leqslant \frac{\frac{1}{p}\int_\Omega|\nabla (tu)|^p\, dx +\frac{1}{2}\int_\Omega |\nabla (tu)|^2\, dx}
{\frac{1}{2}\int_\Omega a(x)(tu)^2 \, dx +\frac{1}{2}\int_{\partial\Omega} b(x)(tu)^2 \, d\sigma}\\
& =\frac{2t^{p-2}}{p}\frac{\int_\Omega|\nabla u|^p\, dx}{\int_\Omega a(x)u^2 \, dx +\int_{\partial\Omega} b(x)u^2 \, d\sigma}
+\frac{\int_\Omega |\nabla u|^2\, dx}{\int_\Omega a(x)u^2 \, dx +\int_{\partial\Omega} b(x)u^2 \, d\sigma}.
\end{align*}
By passing to the limit as $t\to 0$ we deduce that $\lambda_1(p)\leqslant \nu_1(p)$.
\end{proof}

\begin{proposition}\label{prop:lambda1-not_eigenvalue}
$\lambda_1(p)$ is not an eigenvalue of Problem \eqref{eq:eigenvalue_problem-p2_laplacian}.
\end{proposition}

\begin{proof}
Otherwise $\lambda=\lambda_1(p)$ would be an eigenvalue with corresponding eigenfunction $u_\lambda$. By Lemma \ref{lemma:lambda_igual_nu},
\begin{multline}
\int_\Omega|\nabla u_\lambda|^p\, dx +\nu_1(p)\Big(\int_\Omega a(x)u_\lambda^2 \, dx +\int_{\partial\Omega} b(x)u_\lambda^2 \, d\sigma\Big)\\
 \leqslant \int_\Omega|\nabla u_\lambda|^p\, dx +\int_\Omega|\nabla u_\lambda|^2\, dx \\
 = \lambda_1(p)\Big(\int_\Omega a(x)u_\lambda^2\, dx+\int_{\partial\Omega} b(x)u_\lambda^2\, d\sigma\Big),
\end{multline}
which implies $\int_\Omega|\nabla u_\lambda|^p\, dx=0$. By Poincar\'e-Wirtinger inequality $u_\lambda$ should be a constant. This, however, 
is impossible, due to Lemmas \ref{lemma:W1p-igual-Vptil+R} and \ref{lemma:lambda1_positivo} (see also Remark \ref{rmk:eigenfunction_lies_on_Vptil}).
\end{proof}

\begin{proof}[Proof of Theorem \ref{thm:main}\ref{item:p_bigger_2}]
Follows immediately from Lemma \ref{lemma:eigenvalue_is_nonegative} and Propositions \ref{prop:eigenvalues_bigger_than_lambda1}, 
\ref{prop:no-eigenvalue_in_interval} and \ref{prop:lambda1-not_eigenvalue}.
\end{proof}

\section{Proof of Theorem \ref{thm:main}\ref{item:p_less_2}}\label{sec:proof-assertion_b}

In the case $1<p<2$ we have continuous inclusions $W^{1,2}(\Omega)\subset W^{1,p}(\Omega)$, therefore it is natural to analyse Problem 
\eqref{eq:eigenvalue_problem-p2_laplacian} in the space $W^{1,2}(\Omega)$. Moreover we use the 
embeddings $W^{1,r}(\Omega)\hookrightarrow L^{\frac{rN}{N-r}}(\Omega)$ and $W^{1,r}(\Omega)\hookrightarrow L^{\frac{r(N-1)}{N-r}}(\partial\Omega)$ 
with $r=p$. Thus, if $a\in L^{\frac{pN}{(p-2)N+2p}}(\Omega)$ and $b\in L^{\frac{p(N-1)}{(p-2)N+p}}(\Omega)$ then integrals such as 
$\int_\Omega a(x)u^2\, dx$ and $\int_{\partial\Omega} b(x)u^2\, d\sigma$ will be well-defined and good estimates can be obtained. Clearly, these 
conditions are stronger than those in Section \ref{sec:proof-assertion_a} (and only make sense) for $\frac{2N}{N+1}<p<2$. The reader must bear in 
mind that this restriction on $p$ is not necessary under the more restrictive assumptions `$a\in L^\infty(\Omega)$' and 
`$b\in L^\infty(\partial\Omega)$'.

\begin{definition}\label{def:eigenvalue-eigenfunction-p2_laplacian-p_menor_2}
Let $1<p<2$. We call $\lambda\in \mathbb{R}$ an eigenvalue of Problem \eqref{eq:eigenvalue_problem-p2_laplacian} if there exists a non-zero 
$u\in W^{1,2}(\Omega)$ such that Eq. \eqref{eq:eigenvalue-eigenfunction-p2_laplacian} holds for all $\phi\in W^{1,2}(\Omega)$. Such a function 
$u\in W^{1,2}(\Omega)\backslash\{0\}$ will be called an eigenfunction corresponding to the eigenvalue $\lambda$. In other words, $\lambda\in \mathbb{R}$ is an 
eigenvalue of Problem \eqref{eq:eigenvalue_problem-p2_laplacian} with corresponding eigenfunction $u\in W^{1,2}(\Omega)\backslash\{0\}$ if and only 
if $u$ is a critical point of the $C^1$ functional defined in Eq. \eqref{eq:functional-p2_laplacian}. 
\end{definition}

The following result is an immediate consequence of what has been done in the previous section.

\begin{proposition}\label{prop:no_eigenvalues-p_less_2}
Let $\frac{2N}{N+1}<p<2$ and $\mu_1(p)$ be defined by \eqref{eq:first_nonzero_eigenvalue-W12}. Then no number in the 
set $(-\infty,0)\cup (0,\mu_1(p)]$ is an eigenvalue of Problem \eqref{eq:eigenvalue_problem-p2_laplacian}.
\end{proposition}

Actually, as we have mentioned, the hipotheses on $a$ and $b$ here are stronger than those in Section \ref{sec:proof-assertion_a}, in the sense that 
`$a\in L^{\frac{pN}{(p-2)N+2p}}(\Omega)$' implies `$a\in L^{\frac{N}{2}}(\Omega)$' (as far as $1<p<2$), and similarly for $b$; thus, the proof of 
Lemma \ref{lemma:W1p-igual-Vptil+R} is still valid and gives us the decomposition $W^{1,2}(\Omega)=\mathscr{W}_2\oplus \mathbb{R}$; this is what 
we need in the sequel. Lemma \ref{lemma:eigenvalue_is_nonegative} (and its proof) also holds without any change. Besides, the same proof in 
Lemma \ref{lemma:lambda1_positivo} is valid and shows that $\mu_1(p)>0$. The proof of Proposition \ref{prop:no-eigenvalue_in_interval} does 
not work as it stands but can be easily adapted. In fact, if we define 
\begin{equation}\label{eq:nu1}
\nu_1:=\inf_{u\in \mathscr{W}_2\backslash \{0\}}\frac{\int_\Omega |\nabla u|^2\, dx}{\int_\Omega a(x)u^2 \, dx + \int_{\partial\Omega} b(x)u^2} 
\end{equation}
then the same proof in Lemma \ref{lemma:lambda_igual_nu} (except that we take $t\to \infty$ instead) shows that $\mu_1(p)=\nu_1$. Thus, if there was an 
eigenvalue $\lambda\in (0,\mu_1(p))$ with corresponding eigenfunction $u_\lambda\in \mathscr{W}_2\backslash\{0\}$ then we would have
\begin{multline}
0<\frac{\mu_1(p)-\lambda}{2}\Big(\int_\Omega a(x)u_\lambda^2 \, dx + \int_{\partial\Omega} b(x)u_\lambda^2 \, d\sigma\Big)\\
\leqslant \frac{1}{2}\int_\Omega |\nabla u_\lambda|^2\, dx 
-\frac{\lambda}{2}\int_\Omega a(x)u_\lambda^2 \, dx -\frac{\lambda}{2} \int_{\partial\Omega} b(x)u_\lambda^2 \, d\sigma \\
\leqslant \frac{1}{2}\int_\Omega|\nabla u_\lambda|^p\, dx +\frac{1}{2}\int_\Omega |\nabla u_\lambda|^2\, dx \\
-\frac{\lambda}{2}\int_\Omega a(x)u_\lambda^2 \, dx -\frac{\lambda}{2} \int_{\partial\Omega} b(x)u_\lambda^2 \, d\sigma=0,
\end{multline}
which is impossible. Finally, the same proof in Proposition \ref{prop:lambda1-not_eigenvalue} reveals that $\mu_1(p)$ is not an eigenvalue.

It is not clear, however, that the conclusion of Lemma \ref{lemma:limit-coercive-Vptil} holds for $1<p<2$, since the functional $\mathscr{I}_\lambda$ 
given in \eqref{eq:functional-p2_laplacian} is not coercive in this case. From now on we analyse the action of $\mathscr{I}_\lambda$ on the so called Nehari manifold defined, for each $\lambda>\mu_1(p)$, by
\begin{align}
\mathscr{N}_\lambda&:=\{u\in \mathscr{W}_2\backslash \{0\}:\langle \mathscr{I}'_\lambda(u),u\rangle=0\}\\
&=\Big\{u\in \mathscr{W}_2\backslash \{0\}:
\int_\Omega|\nabla u|^p\, dx+\int_\Omega|\nabla u|^2\, dx\\
&\hspace{5cm}=\lambda\int_\Omega a(x)u^2\, dx+\lambda\int_{\partial\Omega} b(x)u^2\, d\sigma\Big\}.
\end{align}
Note that on $\mathscr{N}_\lambda$ the functional $\mathscr{I}_\lambda$ is given by
\begin{equation}\label{eq:I_on_Nehari}
\begin{aligned}
\mathscr{I}_\lambda(u)&=\frac{1}{p}\int_\Omega|\nabla u|^p\, dx +\frac{1}{2}\int_\Omega |\nabla u|^2\, dx 
-\frac{\lambda}{2}\int_\Omega a(x)u^2 \, dx -\frac{\lambda}{2} \int_{\partial\Omega} b(x)u^2 \, d\sigma \\
&=\Big(\frac{1}{p}-\frac{1}{2}\Big)\int_\Omega|\nabla u|^p\, dx .
\end{aligned}
\end{equation}
In particular, $\mathscr{I}_\lambda$ is homogeneous of degree $p$ on $\mathscr{N}_\lambda$ in the sense that 
$\mathscr{I}_\lambda(tu)=t^p\mathscr{I}_\lambda(u)$ for all 
$u\in \mathscr{N}_\lambda$. However, $\mathscr{I}_\lambda$ is not necessarily coercive on $\mathscr{N}_\lambda$ which, otherwise, would facilitate some of our 
labor below. As is well known, the Nehari manifold is a natural constraint for $\mathscr{I}_\lambda$ and we work in the sequel to show that the minimum of 
$\mathscr{I}_\lambda$ restricted to $\mathscr{N}_\lambda$ turns out to be a free critical point, that is, a critical point of $\mathscr{I}_\lambda$ considered on the 
whole space.

In the rest of this paper recall that $\mu_1(p)$ equals $\nu_1$ (cf. Eq. \eqref{eq:nu1}). In what follows, and until further notice, $\lambda>\mu_1(p)$ is 
a fixed real number. First, we observe that the Nehari manifold $\mathscr{N}_\lambda$ is non-empty. In fact, from the definition of $\nu_1$,  
there exists $v_\lambda\in \mathscr{W}_2\backslash \{0\}$ such that 
$$
\int_\Omega |\nabla v_\lambda|^2\, dx <\lambda\int_\Omega a(x)v_\lambda^2 \, dx +\lambda \int_{\partial\Omega} b(x)v_\lambda^2 \, d\sigma ,
$$
thus $tv_\lambda\in \mathscr{N}_\lambda$ for some $t>0$; in fact this is equivalent to the identity
\begin{equation}\label{eq:solve_for_t-belong_to_Nehari}
t^p\int_\Omega|\nabla v_\lambda|^p\, dx+t^2\int_\Omega|\nabla v_\lambda|^2\, dx=
\lambda t^2\int_\Omega a(x)v_\lambda^2\, dx+\lambda t^2\int_{\partial\Omega} b(x)v_\lambda^2\, d\sigma
\end{equation}
which can be explicitly solved for $t$. 

\begin{lemma}\label{lemma:sequence_on_Nehari-bounded_in_W12}
Let $(u_n)\subset \mathscr{N}_\lambda$ be such that $\sup_{n\in \mathbb{N}}\int_\Omega|\nabla u_n|^p\, dx<\infty$. Then $(u_n)$ is bounded in 
$W^{1,2}(\Omega)$.
\end{lemma}

\begin{proof}
Let $(u_n)\subset \mathscr{N}_\lambda$ be as in the statement of Lemma \ref{lemma:sequence_on_Nehari-bounded_in_W12}. In particular,
\begin{equation}\label{eq:sequence_in_Nehari}
\int_\Omega|\nabla u_n|^p\, dx+\int_\Omega|\nabla u_n|^2\, dx =\lambda\int_\Omega a(x)u_n^2\, dx+\lambda\int_{\partial\Omega} b(x)u_n^2\, d\sigma. 
\end{equation}
We split the proof into two steps.

{\it Step 1.} Suppose $\sup_{n\in \mathbb{N}}\|u_n\|_{L^2}<\infty$. As in Lemma \ref{lemma:lambda1_positivo} we can estimate
\begin{align*}
&\int_\Omega|\nabla u_n|^2\, dx\\
&\leqslant \lambda \Big(\int_\Omega a(x)u_n^2\, dx+\int_{\partial\Omega} b(x)u_n^2\, d\sigma\Big)\\
&\leqslant \lambda\Big(C_3\|a\|_{L^{\frac{pN}{(p-2)N+2p}}(\Omega)}+
C_4\|b\|_{L^{\frac{p(N-1)}{(p-2)N+p}}(\partial\Omega)}\Big)\|u_n-\overline{u_n}\|_{W^{1,p}(\Omega)}^2\\
&\leqslant \lambda\Big(C_3\|a\|_{L^{\frac{pN}{(p-2)N+2p}}(\Omega)}+C_4\|b\|_{L^{\frac{p(N-1)}{(p-2)N+p}}(\partial\Omega)}\Big)(1+C^\text{P}_p)
\Big(\int_{\Omega}|\nabla u_n|^p\, dx\Big)^{\frac{2}{p}}.
\end{align*}
We can infer that $\sup_{n\in \mathbb{N}}\int_\Omega|\nabla u_n|^2\, dx<\infty$, thus $(u_n)$ is bounded in $W^{1,2}(\Omega)$ in this case.

{\it Step 2.} Suppose (after passing to a subsequence if necessary) that $\|u_n\|_{L^2(\Omega)}\to \infty$ as $n\to \infty$. Put 
$v_n:=\frac{u_n}{\|u_n\|_{L^2(\Omega)}}$. As in Step 1 above we can deduce that $(v_n)\subset \mathscr{W}_2$ is bounded in $W^{1,2}(\Omega)$. Thus there 
exists a $v_0\in \mathscr{W}_2$ such that $v_n\rightharpoonup v_0$ in $W^{1,2}(\Omega)$ (also in $W^{1,p}(\Omega)$, by continuous inclusion) 
and $v_n\to v_0$ in $L^{2}(\Omega)$.

Dividing \eqref{eq:sequence_in_Nehari} by $\|u_n\|_{L^2(\Omega)}^p$ we find
$$
\int_{\Omega}|\nabla v_n|^p\, dx =
\frac{\lambda\int_\Omega a(x)u_n^2\, dx+\lambda\int_{\partial\Omega} b(x)u_n^2\, d\sigma-\int_\Omega|\nabla u_n|^2\, dx}{\|u_n\|_{L^2(\Omega)}^p}
\to 0,\ \ \ \text{as}\ n\to \infty .
$$
Since $v_n\rightharpoonup v_0$ in $W^{1,p}(\Omega)$ we have
$$\int_{\Omega}|v_0|^p\, dx+\int_{\Omega}|\nabla v_0|^p\, dx\leqslant 
\liminf_{n\to \infty} \Big(\int_{\Omega}| v_n|^p\, dx+\int_{\Omega}|\nabla v_n|^p\, dx\Big),$$
which implies, since $v_n\to v_0$ in $L^{p}(\Omega)$, that
$$\int_{\Omega}|\nabla v_0|^p\, dx\leqslant \liminf_{n\to \infty} \int_{\Omega}|\nabla v_n|^p\, dx=0.$$
This, in combination with Poincaré-Wirtinger inequality, implies that $v_0$ is constant. In view of Lemma \ref{lemma:W1p-igual-Vptil+R} this constant is zero. Therefore we find that 
$v_n\to 0$ in $L^{2}(\Omega)$ but this contradicts the fact that $\|v_n\|_{L^2(\Omega)}=1$ for all $n\in \mathbb{N}$. Therefore $(u_n)$ must be bounded in 
$L^2(\Omega)$ and we are back to Step 1 above.
\end{proof}

\begin{lemma}
$m=\inf_{w\in \mathscr{N}_\lambda}\mathscr{I}_\lambda(w)>0$ and $m=\mathscr{I}_\lambda(u)$ for some $u\in \mathscr{N}_\lambda$.
\end{lemma}

\begin{proof}
We split the proof into two steps.

{\it Step 1.} First we show that $m>0$. Otherwise, suppose $m=0$ and let $(u_n)\subset \mathscr{N}_\lambda$ be a minimizing sequence, so that 
$\mathscr{I}_\lambda(u_n)\to 0$ as $n\to \infty$. From \eqref{eq:I_on_Nehari} we can infer that
\begin{multline}\label{I-minimizing_on_Nehari-goes-to-zero}
0\leqslant \frac{\lambda}{2}\int_\Omega a(x)u_n^2\, dx + \frac{\lambda}{2}\int_{\partial\Omega} b(x)u_n^2\, d\sigma-\frac{1}{2}\int_\Omega|\nabla u_n|^2\, dx\\ 
=\frac{1}{p}\int_\Omega|\nabla u_n|^p\, dx \to 0\ \ \ \text{as}\ n\to \infty .
\end{multline}
By Lemma \ref{lemma:sequence_on_Nehari-bounded_in_W12}, $(u_n)$ is bounded in $W^{1,2}(\Omega)$, thus $u_n\rightharpoonup u_0$ in 
$W^{1,2}(\Omega)$ (and also weakly in $W^{1,p}(\Omega)$) and $u_n\to u_0$ in $L^{2}(\Omega)$ (and in $L^p(\Omega)$) for some $u_0\in \mathscr{W}_2$. But then
$$\int_{\Omega}|\nabla u_0|^p\, dx\leqslant \liminf_{n\to \infty} \int_{\Omega}|\nabla u_n|^p\, dx=0,$$
and, as in the proof Lemma \ref{lemma:sequence_on_Nehari-bounded_in_W12}, we can conclude that $u_0=0$. Moreover, we can deduce that, for the sequence 
$v_n:=\frac{u_n}{\|u_n\|_{L^2(\Omega)}}$, there exists $v_0\in \mathscr{W}_2$ such that $v_n\rightharpoonup v_0$ in $W^{1,2}(\Omega)$ and in 
$W^{1,p}(\Omega)$, and $v_n\to v_0$ in $L^{2}(\Omega)$. Dividing Eq. \eqref{I-minimizing_on_Nehari-goes-to-zero} by $\|u_n\|_{L^2(\Omega)}^p$ we find
\begin{multline}
\frac{1}{p}\int_\Omega|\nabla v_n|^p\, dx=
\|u_n\|_{L^2(\Omega)}^{2-p}\Big(\frac{\lambda}{2}\int_\Omega a(x)v_n^2\, dx + \frac{\lambda}{2}\int_{\partial\Omega} b(x)v_n^2\, d\sigma\\
-\frac{1}{2}\int_\Omega|\nabla v_n|^2\, dx \Big)\to 0\ \ \ \text{as}\ n\to \infty
\end{multline}
since the expression between parentheses is bounded. Again we can deduce that $v_0=0$, which is absurd.

{\it Step 2.} Now we show that $m=\mathscr{I}_\lambda(u)$ for some $u\in \mathscr{N}_\lambda$. Let $(u_n)\subset \mathscr{N}_\lambda$ be a minimizing sequence, 
so that $\mathscr{I}_\lambda(u_n)\to m$ as $n\to \infty$. In particular, the sequence $(u_n)$ satisfies Eq. \eqref{eq:sequence_in_Nehari} and is bounded in 
$W^{1,2}(\Omega)$ by Lemma \ref{lemma:sequence_on_Nehari-bounded_in_W12}, so that $u_n\rightharpoonup u_1$ in $W^{1,2}(\Omega)$ 
(and in $W^{1,p}(\Omega)$) and $u_n\to u_1$ in $L^{2}(\Omega)$ for some element $u_1\in \mathscr{W}_2$. We claim $u_1\in \mathscr{N}_\lambda$ and 
$\mathscr{I}_\lambda(u_1)=m$. By passing to limit as $n\to \infty$ in Eq. 
\eqref{eq:sequence_in_Nehari} we find
\begin{equation}
\int_\Omega|\nabla u_1|^p\, dx+\int_\Omega|\nabla u_1|^2\, dx \leqslant\lambda\int_\Omega a(x)u_1^2\, dx+\lambda\int_{\partial\Omega} b(x)u_1^2\, d\sigma. 
\end{equation}
Moreover, $u_1\neq 0$; otherwise we would have $u_n\to 0$ in $L^2(\Omega)$ and Eq. \eqref{eq:sequence_in_Nehari} would imply 
$\int_{\Omega}|\nabla u_n|^p\, dx\to 0$ which, as in Step 1, would lead to a contradiction. 

If identity holds in the inequality above (and we claim it does) then we have $u_1\in \mathscr{N}_\lambda$ and the proof is complete. Otherwise, that is, 
if strict inequality holds in the above inequality, then we have $tu_1\in \mathscr{N}_\lambda$ for some $t\in (0,1)$; in fact, such a number can 
be obtained by solving Eq. \eqref{eq:solve_for_t-belong_to_Nehari} with $v_\lambda$ replaced by $u_1$ and, from its explicit expression, the 
aforementioned strict inequality garantees that it lies between 0 and 1. But then
$$0<m\leqslant \mathscr{I}_\lambda (tu_1)=t^p\mathscr{I}_\lambda(u_1)\leqslant t^p\liminf_{n\to \infty}\mathscr{I}_\lambda(u_n)=t^pm<m ,$$
which is a contradiction.
\end{proof}

\begin{proposition}\label{prop:eigenvalues-p_less_2}
Every $\lambda\in (\mu_1(p),\infty)$ is an eigenvalue of Problem \eqref{eq:eigenvalue_problem-p2_laplacian}.
\end{proposition}
\begin{proof}
Fix $\lambda> \mu_1(p)$. Let $u\in \mathscr{N}_\lambda$ be such that $\mathscr{I}_\lambda(u)=m$. In particular
\begin{equation}
\int_\Omega|\nabla u|^2\, dx <\lambda\int_\Omega a(x)u^2\, dx+\lambda\int_{\partial\Omega} b(x)u^2\, d\sigma.
\end{equation}
We claim that for every $v\in \mathscr{W}_2$ there exists $\delta>0$ such that 
\begin{equation}
\int_\Omega|\nabla (u+sv)|^2\, dx <\lambda\int_\Omega a(x)(u+sv)^2\, dx+\lambda\int_{\partial\Omega} b(x)(u+sv)^2\, d\sigma
\end{equation}
for all $s\in (-\delta ,\delta)$. In fact, the inequality holds for $s=0$ and both sides above are continuous functions of $s$. Now, by solving 
Eq. \eqref{eq:solve_for_t-belong_to_Nehari} with $v_\lambda$ replaced by $u+sv$ we are able to find $t(s)>0$ satisfying 
$t(s)(u+sv)\in \mathscr{N}_\lambda$ for all $s\in (-\delta ,\delta)$. Besides, $t(s)$ is differentiable (this can be seen from the explicit 
expression for $t(s)$ after solving Eq. \eqref{eq:solve_for_t-belong_to_Nehari}) and $t(0)=1$.

Obviously, the map $\gamma:(-\delta ,\delta)\to \mathbb{R}$ defined by
$$\gamma(s):=\mathscr{I}_\lambda(t(s)(u+sv))$$
belongs to $C^1(-\delta,\delta)$, satisfies $\gamma(0)\leqslant \gamma(s)$ for all $s\in (-\delta,\delta)$, and then
$$0=\gamma'(0)=\langle \mathscr{I}_\lambda'(t(0)u),t'(0)u+t(0)v\rangle =\langle \mathscr{I}_\lambda(u),v\rangle .$$
Therefore $\lambda$ is an eigenvalue.
\end{proof}

\begin{proof}[Proof of Theorem \ref{thm:main}\ref{item:p_less_2A}]
As it was already pointed out, $\lambda=0$ is an eigenvalue. Therefore, the conclusion follows immediately from Propositions 
\ref{prop:no_eigenvalues-p_less_2} and \ref{prop:eigenvalues-p_less_2}.
\end{proof}

Now we turn our attention to the proof of Assertions \ref{item:p_less_2B} and \ref{item:p_less_2C}. Let us start by observing that as 
$p\downarrow \frac{2N}{N+1}$ then the integrability exponent attached to $b$ blows up. The same applies to $a$ 
when $p\downarrow \frac{2N}{N+2}$. An inspection in the proofs in this section reveals that the only point that needs to be addressed here 
is Step 1 in the proof of Lemma \ref{lemma:sequence_on_Nehari-bounded_in_W12}. To carry out the necessary estimates we use the following 
well-known estimate: for all $\varepsilon>0$ there exists a constant $c_\varepsilon\geqslant 0$ such that
\begin{equation}\label{eq:estimate-trace2-lessL2+grad2}
 \int_{\partial\Omega} u^2\, d\sigma\leqslant 
 \varepsilon \int_\Omega|\nabla u|^2\, dx+c_\varepsilon\int_\Omega u^2\, dx \ \ \ (u\in W^{1,2}(\Omega)).
\end{equation}
It can be proved either indirectly, first for smooth functions $u\in C^1(\overline{\Omega})$ (see e.g. \cite[p. 177]{friedman1983partial}) and then 
for general elements $u\in W^{1,2}(\Omega)$ by approximation, or directly by invoking the compactness of the 
trace (see e.g. \cite[Lemma 1]{afrouzi1999principal}).

\begin{proof}[Proof of Theorem \ref{thm:main}\ref{item:p_less_2B}]
From \eqref{eq:sequence_in_Nehari} we have
\begin{align*}
&\int_\Omega|\nabla u_n|^2\, dx\\
&\leqslant \lambda \Big(\int_\Omega a(x)u_n^2\, dx+\int_{\partial\Omega} b(x)u_n^2\, d\sigma\Big)\\
&\leqslant \lambda\Big(\|a\|_{L^{\frac{pN}{(p-2)N+2p}}(\Omega)}\Big(\int_\Omega |u_n|^{\frac{pN}{N-p}}\, dx\Big)^{\frac{2(N-p)}{pN}}
+\|b\|_{L^{\infty}(\partial\Omega)}\int_{\partial\Omega}|u_n|^2\, d\sigma\Big).
\end{align*}
Since $p>\frac{2N}{N+2}$, $L^{\frac{pN}{N-p}}(\Omega)$ embedds into $L^2(\Omega)$ which, in combination with estimate 
\eqref{eq:estimate-trace2-lessL2+grad2}, allows us to estimate $\int_\Omega|\nabla u_n|^2\, dx$ by $\int_\Omega|u_n|^2\, dx$.
\end{proof}

\begin{proof}[Proof of Theorem \ref{thm:main}\ref{item:p_less_2C}]
From \eqref{eq:sequence_in_Nehari} we have
\begin{align*}
\int_\Omega|\nabla u_n|^2\, dx&\leqslant \lambda\Big(\|a\|_{L^{\infty}(\Omega)}\int_{\Omega}|u_n|^2\, dx
+\|b\|_{L^{\infty}(\partial\Omega)}\int_{\partial\Omega}|u_n|^2\, d\sigma\Big).
\end{align*}
Then proceed as in the previous proof.
\end{proof}

We note as a curious fact that the above proofs does not require the hypothesis `$\sup_{n\in \mathbb{N}}\int_\Omega |\nabla u_n|^p\, dx<\infty$'. 
Actually, the same is true for Step 1 in the proof of Lemma \ref{lemma:sequence_on_Nehari-bounded_in_W12} itself, since 
$L^{\frac{p(N-1)}{N-p}}(\partial\Omega)$ embedds into $L^2(\partial\Omega)$ for $p>\frac{2N}{N+1}$.

\bibliographystyle{plain}
\bibliography{biblio-eigenvalue}

\end{document}